\documentclass[12pt]{elsarticle}
\usepackage{amssymb}

\usepackage{amsfonts,amsmath,amsthm}
\usepackage{amssymb}
\usepackage[utf8]{inputenc}
\usepackage[english]{babel}
\usepackage{graphicx}
\usepackage{epstopdf}
\usepackage{tikz}
\usetikzlibrary{shapes,arrows,backgrounds}
\usepackage{subcaption}
\usepackage{mathtools}
\usepackage{array}
\usepackage{url}
\usepackage{hyperref}
\usepackage{booktabs}

\newtheorem{theorem}{Theorem}[section]
\newtheorem*{theorem-non}{Theorem}
\newtheorem{lemma}[theorem]{Lemma}

\newtheorem{proposition}[theorem]{Proposition}
\newtheorem*{proposition-non}{Proposition}

\newtheorem{corollary}[theorem]{Corollary}

\newtheorem{definition}[theorem]{Definition}

\newtheorem*{conjecture-non}{Conjecture}
\theoremstyle{remark}

\newtheorem{questions}[theorem]{Questions}

\newtheorem{nota}[theorem]{}
\newtheorem{claim}{Claim}
\newtheorem{remark}[theorem]{Remark}

\def\cl{\overline}

\def\DD{\operatorname{D}}
\def\ND{\operatorname{ND}}
\def\NDif{\operatorname{NDif}}
\def\R{\mathbb R}
\def\E{\mathcal{E}}
\def\Z{\mathbb Z}
\def\N{\mathbb N}
\def\B{\mathcal B}
\def\O{o}
\def\C{\mathcal C}
\def\e{\varepsilon}
\def\norma{\|\cdot\|}
\def\nX{\|\cdot\|_X}
\def\XnX{(X,\|\cdot\|_X)}

\def\RRX{(\R^2,\|\cdot\|_X)}
\def\nY{\|\cdot\|_Y}
\def\YnY{(Y,\|\cdot\|_Y)}
\def\tauSXY{\tau:S_X\to S_Y}

\title{Linearity of isometries between convex Jordan curves}

\begin{document}

\author
{Javier~Cabello~Sánchez}
\address{
Departamento de Matem\'{a}ticas, Universidad de Extremadura, 
Avda. de Elvas s/n, 06006 Badajoz. Spain. {coco@unex.es}}

\address{\em Supported in part by DGICYT projects MTM2016-76958-C2-1-P 
and PID2019-103961GB-C21 (Spain) and Junta de Extremadura programs GR-15152 and IB-16056.} 

\address{\em Keywords: Tingley's Problem; differentiability; finite-dimensional spaces; 
metric invariants.}
\address{\em 2020 Mathematics Subject Classification: 46B04, 15A03, 52A10.}

\begin{abstract}
In this paper, we show that the $C^1$-differentiability of the norm of a 
two-dimensional normed space depends only on distances between points of the 
unit sphere in two different ways. 

As a consequence, we see that any isometry between the spheres of normed planes 
$\tau:S_X\to S_Y$ is linear, provided that there exist linearly independent 
$x,\overline{x}\in S_X$ where $S_X$ is not differentiable and that $S_X$ is piecewise 
differentiable. 

We end this work by showing that the isometry $\tau:C_X\to C_Y$ is linear even 
if it is not an isometry between spheres: every isometry between (planar) Jordan 
piecewise $C^1$-differentiable convex curves extends to $X$ whenever $X$ and $Y$ 
are strictly convex and the amount of non-differentiability points of $S_X$ and $S_Y$ 
is finite and greater than 2. 
\end{abstract}

\maketitle

\section{Introduction}

The study of isometries between Banach spaces led, back in the 30's, to one of 
the best known results in Functional Analysis, the Mazur--Ulam Theorem. This 
result, see~\cite{MazurUlam}, states that every onto isometry between two 
Banach spaces is affine. So, if an onto isometry preserves the origin, then 
the isometry is linear. Forty years later, P.~Mankiewicz (\cite{Mankiewicz}) 
proved that every onto isometry between convex bodies in two Banach spaces 
is also affine. The foreseeable generalisation of these results is 
{\em Every onto isometry between the spheres of two Banach spaces is linear,} 
and this could be ultimately generalised as 
{\em Every onto isometry between the boundaries of convex bodies 
of two Banach spaces is affine,}
but, up to now, no-one has been able to prove or disprove this statement. 
This innocent-looking problem was stated in 1987 by D.~Tingley (\cite{Tingley}), 
but it turns out to be way more challenging than it could seem at first glance. 
Tingley's Problem has evolved to that of extending isometries between spheres to 
isometries between the whole spaces, and the greatest advances have been achieved 
when both spaces have some common structure, such as von Neumann algebras, 
trace class operators spaces, sums of strictly convex spaces\ldots This 
{\em extension of isometries problem} has 
experienced a rapid development in the last few years, and there are lots of 
kinds of spaces where Tingley's Problem has a positive answer, see  
\cite{TBanakhC2, CP18, ding2013, 
fangwang2006, FPeralta2018vN, PeraltaNG, FPeralta2017, FPeralta2018Low, 
KadetsMiguel, li2016, 
Mori18, MoriOzawa19, Peralta2018Acta, Peralta2019RevMC, 
PeraTa2019, TanXiong, tanaka2014, Tanaka2016matr, 
Tanaka2017vN, wanghuang}. 

Nevertheless, there is another way to look at this Problem. Instead of 
extending an isometry, one can rule out the existence of an isometry between 
two spheres --a trivial example: there is no isometry between the spheres of 
$(\R^2,\norma_2)$ and $(\R^2,\norma_\infty)$, say $S_2$ and $S_\infty$, because 
there exist $x,y,z\in S_\infty$ 
such that $\|x-y\|_\infty=\|x-z\|_\infty=\|y-z\|_\infty=2$ but this cannot 
happen in $S_2$. To the best of our knowledge, 
the first great achievement in 
this setting can be found in~\cite{KadetsMiguel}, where the authors prove that, 
in finite-dimensional spaces,  
no sphere can be isometric to a polyhedral sphere unless it is polyhedral, too. 
Actually, they also extend the isometry between the spheres, so the main result 
in~\cite{KadetsMiguel} is 

{\em If $X$ is finite-dimensional and polyhedral and there is an onto isometry 
$\tauSXY$, then $Y$ is also polyhedral and $X$ and $Y$ are linearly isometric.}

In the same spirit, a few years later appeared this result: 

{\em If $X$ is an inner product space and there is an onto isometry $\tauSXY$, then 
$Y$ is also an inner product space and $X$ and $Y$ are linearly isometric,} 
see \cite{BCFP18, JCSRefJMAA, MoriOzawa19}. 

So, motivated by the huge advance that can be seen at a recent paper by Tarás 
Banakh, \cite{TBanakhC2}, whose main result is 

{\em Every isometry $\tauSXY$ between the spheres of 
absolutely smooth two-dimensional spaces is linear}, \\ 
we began to study whether the $C^2$-differentiability (that implies absolute 
differentiability) of some two-dimensional 
normed space $\XnX$ can be expressed in terms of $(S_X,\nX)$ --unsuccessfully. 

However, we have been able to determine $C^1$-differentiability in two independent 
ways. The first way is easily seen to hold in finite-dimensional spaces, whereas we 
have not been able to prove whether the second one works in dimensions higher than 2 or not. 

These two ways are the following: 
\begin{enumerate}
\item Given a finite-dimensional normed space $\XnX$, the norm $\nX$ fails to 
be differentiable at $x$ if and only if there exist $\alpha, \e_0>0$ such that 
for every $\e<\e_0$ there exist $u,v\in S_X\!\setminus\!\{\pm x\}$ such that
$$
\max\{\|u-x\|_X, \|v+x\|_X\}\leq\e,\qquad \|u-v\|_X\leq 2-\alpha\e.
$$
\item Given $x,y,z\in S_X$ such that $\|\cdot\|_X$ is differentiable at $z$ 
and $z-y=\lambda x$ for some $\lambda>0$, the norm $\|\cdot\|_X$ is 
differentiable at $x\in S_X$ if and only if 
$$G(t)=\|\gamma_z(t)-y\|$$ 
is differentiable at 0, where 
$\gamma_z:\R\to S_X$ is an arc-length parameterization such that $\gamma_z(0)=z$. 
\end{enumerate}

With these two facts in mind, it is not too difficult to show the Tingley-type 
result in this paper: 

\begin{theorem}\label{ElT}
Let $\XnX$ and $\YnY$ be two-dimensional normed spaces and let $\tauSXY$ be 
an isometry. If the amount of points in $S_X$ where $\nX$ is not 
differentiable is finite and greater than 3, then $\tau$ is linear. 
\end{theorem}

With the same ideas as in the proof of Theorem~\ref{ElT}, we have also been 
able to show this result about, so to say, boundaries of convex open subsets 
of $\R^2$: 

\begin{theorem}\label{ElotroT}
Let $\XnX$ be a strictly convex normed plane such that the amount of points 
in $S_X$ where $\nX$ is not differentiable is finite and greater than 3.
Let $C_X\subset X$ be a piecewise $C^1$ Jordan curve that encloses a convex set. 
If there is some isometry $C_X\to C_Y$ for some piecewise $C^1$ curve 
$C_Y\subset Y$, being $\YnY$ another normed plane that fulfils the same 
that $\XnX$, then $\tau$ is affine and, so, $X$ and $Y$ are isometric. 
\end{theorem}

\subsection{Notations and background}

\begin{remark}
Given some Jordan curve $C\subset\R^2$, we will say that $C$ is a convex curve 
when it encloses a convex region. 

For any convex piecewise smooth curve $C\subset\R^2$, it is known that there 
are parameterizations $\gamma:\R \to C$ that are smooth at $t$ if and only 
if $C$ is smooth at $\gamma(t)$ and have one-sided derivatives $\gamma_-'(t)$ 
and $\gamma_+'(t)$ at every other $t\in\R $. We will only consider these 
parameterizations. 

As we will heavily use the arc-length anticlockwise parameterization of $S_X$ 
beginning at some point, we will denote this curve in a special way: 

If $z$ belongs to some piecewise $C^1$-differentiable, convex, Jordan curve 
$C\subset\R^2$, then $\gamma_z:\R \to C$ will denote the only anticlockwise 
parameterization of $C$ that fulfils $\gamma_z(0)=\gamma_z(L)=z$, is injective 
when restricted to $[0,L)$, is $L$-periodic and has $\|\gamma_{z,-}'(t)\|_X=
\|\gamma_{z,+}'(t)\|_X=1$ for every $t\in\R $. This parameterization is 
also known as {\em the natural parameterization} of $C$, see \cite{TBanakhC2}. 
\end{remark}

\begin{definition}
Let $(X,\nX)$ be a normed space. We say that $x$ is Birkhoff orthogonal to $y$, 
denoted as $x\perp_B y$, if $\|x+\lambda y\|_X\geq \|x\|_X$ for every 
$\lambda\in\R$. We will denote $x^\perp=\{y\in X:x\perp_B y\}$. 
\end{definition}

The reader interested in this and related concepts may want to take a look 
at~\cite{alonso2012birkhoff}. It is noteworthy that the main subjects 
of~\cite{alonso2012birkhoff} are two kinds of orthogonality, one of them 
is obviously preserved by isometries of the sphere but we will deal 
with the other one. 

For our particular concern, Birkhoff orthogonality is important due to the 
following two results. 

\begin{proposition}\label{contieneH}{\em\cite[Theorem 2.2]{james1947}, 
\cite[Theorem 4.12]{alonso2012birkhoff}}
For any vector $x$ in a normed linear space $X$ there exists a hyperplane 
$H\subset X$ such that $x\perp_B H$. 
\end{proposition}

\begin{proposition}\label{esH}{\em\cite[Theorem 4.2]{james1947},\ 
\cite[Theorem 4.15]{alonso2012birkhoff}}
The norm of a normed linear space $X$ is Gâteaux differentiable at 
$x\in X\setminus\{0\}$ if and only if $x^\perp$ is a hyperplane. 
\end{proposition}

As for the differentiability of finite-dimensional norms, 
joining \cite[Theorem 25.2 and Corollary 25.5.1]{Rock} we obtain: 

\begin{proposition}\label{proprock}
Let $f$ be a convex function on an open convex set $A\subset\R^d$. If $f$ 
has all partial derivatives at each point of $A$, then $f\in C^1(A)$. 
\end{proposition}

Proposition~\ref{proprock} implies that the usual differences between the various 
kinds of differentiability do not exist when we deal with a convex function like 
$\nX:\R^n\to\R$. In particular, 

\begin{lemma}\label{GatoBir}
Let $(X,\norma_X)$ be a finite-dimensional normed space. Then, the following 
conditions are equivalent to one another: 
\begin{itemize}
\item $\nX$ is $C^1$-differentiable. 
\item $\nX$ is Fréchet differentiable. 
\item $\nX$ is Gâteaux differentiable. 
\item $S_X$ is a differentiable manifold. 
\item For each $x\in X, x\neq 0$, $x^\perp$ is a hyperplane. 
\item If, in addition, $X$ is two-dimensional, then the above conditions are 
equivalent to the fact that for every $x\in S_X, t\in \R $, the equality 
$\gamma'_{x,-}(t)=\gamma'_{x,+}(t)$ holds. 
\end{itemize}
\end{lemma}

We will also use this Lemma that Professor Javier Alonso gifted me some years ago: 

\begin{lemma}[J.~Alonso]\label{Alonso}
Let $\RRX$ be a two-dimensional normed space and $x\in S_X$. If $y\in S_X$ 
is a side derivative of the natural parameterization of $S_X$ at $x$, then 
$x\perp_B y$. 
\end{lemma}

\begin{proof}
We need to show that $\|x+\lambda y\|_X\geq 1$ for every $\lambda\in\R$, with 
$$y=\lim_{t\to 0^+}\frac{\gamma_x(t)-x}{t}.$$
We have the following: 
\begin{equation}
\begin{aligned}
\|x+\lambda y\|_X= &
\left\|x+\lambda \lim_{t\to 0^+}\frac{\gamma_x(t)-x}{t}\right\|_X =
\lim_{t\to 0^+} \left\|x+\lambda \frac{\gamma_x(t)-x}{t}\right\|_X = \\
\lim_{t\to 0^+} & \left\|
\frac{\lambda}{t}\gamma_x(t)+\left( 1-\frac{\lambda}{t} \right) x
\right\|_X \geq 
\lim_{t\to 0^+}
\left| \left\| \frac{\lambda}{t}\gamma_x(t) \right\|_X -
\left\| \left( 1-\frac{\lambda}{t} \right) \right\|_X \right|= \\
\lim_{t\to 0^+} &
\left| \left| \frac{\lambda}{t} \right|- \left| 1-\frac{\lambda}{t} \right| \right| 
=1, \text{ for every } \lambda\in\R. 
\end{aligned}
\end{equation}
\end{proof}

\section{Main results}

We will prove that the differentiability of $\nX$ at some $x$ depends on the 
infinitesimal metric structure of $S_X$ around $x$. Later, we will show that 
it can be determined by means of computations carried away far from $x$. 
Joining both facts we will arrive at our main results after some extra work. 

\begin{proposition}\label{Tlocal} 
Let $(X,\nX)$ be a finite-dimensional space. The differentiability of $\nX$ at any 
$x\in S_X$ depends only on the metric structure of the unit sphere $(S_X,\nX)$ 
near $x$ and $-x$. Namely, $\nX$ fails to be differentiable at $x$ if and only 
if there exist $\delta, \e_0>0$ such that for every $\e<\e_0$ there exist 
$u,v\in S_X\!\setminus\!\{\pm x\}$ such that
\begin{equation}\label{uv}
\max\{\|u-x\|_X, \|v+x\|_X\}\leq\e,\qquad \|u-v\|_X\leq 2-\delta\e.
\end{equation}
\end{proposition}

\begin{proof}
Let $x\in S_X$. 

It is clear that if $\nX$ is differentiable at $x$ then for every 
$\delta, \e_0>0$ there is $0<\e<\e_0$ such that (\ref{uv}) cannot hold
for every $u,v\in S_X\!\setminus\!\{\pm x\}.$

If $\nX$ is not differentiable at $x$, then Proposition~\ref{contieneH} and 
Lemma~\ref{GatoBir} imply that $x^\perp$ contains strictly a hyperplane, so 
there is some pair of independent vectors $y,z\in S_X\cap x^\perp$ such that 
$x\in\operatorname{span}\{y,z\}$. We are going to show that there are $u,v\in 
S_X\cap\operatorname{span}\{y,z\}$ that fulfil (\ref{uv}), so we may suppose 
that $X$ is two-dimensional and $X=\operatorname{span}\{y,z\}$. Taking any 
orientation on $X$, we may define $\gamma_x$. As the side 
derivatives of $\gamma_x$ at $0$ are different and fulfil 
$x\perp_B\gamma'_{x,-}(0)$ and $x\perp_B\gamma'_{x,+}(0)$ (Lemma~\ref{Alonso}), 
we may suppose $y=\gamma'_{x,+}(0)$, $z=\gamma'_{x,-}(0)$. 

Taking into account that $x\perp_B y$ if and only if $x\perp_B -y$, we may 
suppose that there exist $\lambda, \mu>0$ such that $x=-\lambda y+\mu z$. 
Consider the basis $\B=\{x,y\}$. Taking coordinates with respect to $\B$, we 
have $z=(z_1,z_2)$ and $z_1=1/\mu, z_2=\lambda/\mu>0$. 
By the very definition of Birkhoff orthogonality, $x\perp_B y$ implies 
$\|(1,t)\|_X=\|x+ty\|_X\geq 1$ and $-x\perp_B z$ implies 
$\|(-1+tz_1/z_2,t)\|_X=\|-x+tz/z_2\|_X\geq 1$ for every $t\in\R$. It is clear 
that, moreover, $\|(\alpha,t)\|_X\geq \alpha$ and 
$\|(-\alpha+tz_1/z_2,t)\|_X\geq \alpha$ for every $\alpha>0$, so we have 
$$B_X\subset\{(\alpha,\beta)\in\R^2:\alpha\leq 1,\ 
\beta z_1/z_2\leq\alpha+1\}.$$
Furthermore, $B_X$ contains the convex hull of $\{(-1,0),(0,1),(1,0)\}$. 
On the one hand, this means that $\nX\leq\norma_1$. On the other hand, 
this implies that for each $t\in\,]0,1[$ the line $\{(\alpha,t):\alpha\in\R\}$ 
intersects with $S_X$ at exactly two points $u=(a^+(t),t),v=(a^-(t),t)$, with 
$$-1+tz_1/z_2\leq a^-(t)\leq -1+t,\quad 1-t\leq a^+(t)\leq 1.$$ 
Thus, we obtain 
$\|u-v\|_X=\|(a^+(t)-a^-(t),0)\|_X\leq 2-tz_1/z_2$ and 
$$\|u-x\|_X=\|(a^+(t)-1,t)\|_X\leq |a^+(t)-1|+|t|= 1-a^+(t)+t\leq 2t,$$
$$\|v+x\|_X=\|(a^-(t)+1,t)\|_X\leq |a^-(t)+1|+|t|= 1+a^-(t)+t\leq 2t.$$
The last three inequalities end the proof, taking $\e=2t$ and $\delta=z_1/(2z_2)$. 
\end{proof}

\begin{corollary}
Let $\XnX, \YnY$ be finite-dimensional normed spaces whose spheres are 
isometric. Then, $\nY$ is differentiable if and only if $\nX$ is also differentiable. 
\end{corollary}

\begin{proof}
It is straightforward from Proposition~\ref{Tlocal} and the fact that every 
onto isometry between finite-dimensional spheres preserves antipodes 
(\cite[Theorem, p. 377]{Tingley}). 
\end{proof}

\begin{corollary}\label{Clocal}
Let $\XnX, \YnY$ be two-dimensional normed spaces and $\tauSXY$ a surjective 
isometry between their spheres. Then, either both spaces are piecewise 
$C^1$-differentiable or none of them is. Moreover, if $\nX$ is piecewise 
$C^1$ then it is $C^1$-differentiable at $x$ if and only if 
$\nY$ is $C^1$-differentiable at $\tau(x)$. 
\end{corollary}

\begin{remark}
To avoid confusion, we will use the notation $]\alpha,\beta[$ to denote 
the open interval whose endpoints are $\alpha$ and $\beta$. Thus, 
$(\alpha,\beta)$ will always be a vector in $\R^2$. 
\end{remark}

\begin{proposition}\label{Lejos}
Let $\nX$ be a strictly convex norm defined on $X=\R^2$. Consider $x,y,z\in S_X,$ 
and $\lambda\in\,]0,2[$ such that $z=y+\lambda x$ and $\nX$ is differentiable 
at $z$. In these conditions, $\nX$ is differentiable at $x$ if and only if 
$$G(t)=\|\gamma_z(t)-y\|_X$$
is differentiable at $t=0$. In particular, the differentiability at $x$ 
depends on the metric at $y$ and around $z$. 
\end{proposition}

\begin{proof}
It is clear that if $\nX$ is differentiable at $x$ and $z$, then $G$ is 
differentiable at 0 because it is the composition of differentiable functions. 

Suppose, on the other hand, that $\nX$ is not differentiable at $x$, i.e., 
$\gamma_{x,+}'(0)\neq \gamma_{x,-}'(0)$. 
For the sake of clarity, we will consider the basis $\B=\{-\gamma_{x,-}'(0),x\} $
so the position of $x,y,z$ is like in Figure~\ref{nosuave}, i.e., 
$x=(0,1)$ and $z-y=(0,\lambda)$. 

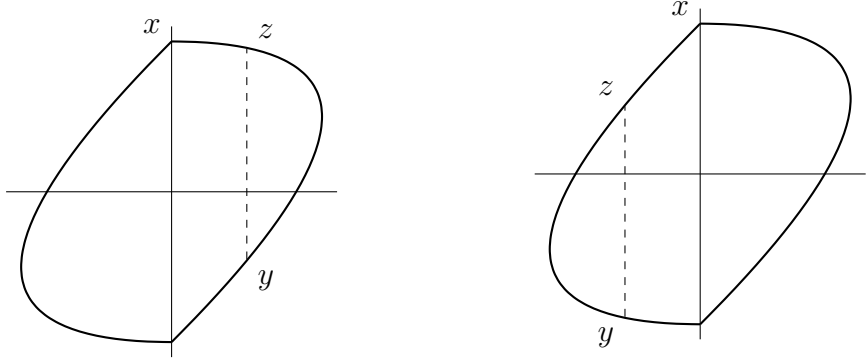
\begin{figure}[ht]
\begin{center}
\begin{tikzpicture}
  \draw[-] (-2.2, 0) -- (2.2, 0); 
  \draw[-] (0, -2.2) -- (0, 2.2) node[left] {$x$}; 
  \draw[dashed] (1.0, -0.90) node[below right] {$y$} -- (1.0, 1.90) node[above right] {$z$};
  \draw[scale=2.0, domain=-1.0:1.0, smooth, thick, variable=\y, black] plot ({\y*\y-1}, {\y+\y*\y/2-1/2});
  \draw[scale=2.0, domain=-1.0:1.0, smooth, thick, variable=\y, black] plot ({1-\y*\y}, {1/2-\y-\y*\y/2});
\end{tikzpicture}
\qquad
\qquad
\qquad
\begin{tikzpicture}
  \draw[-] (-2.2, 0) -- (2.2, 0); 
  \draw[-] (0, -2.2) -- (0, 2.2) node[left] {$x$}; 
  \draw[dashed] (-1.0, -1.90) node[below left] {$y$} -- (-1.0, 0.90) node[above left] {$z$};
  \draw[scale=2.0, domain=-1.0:1.0, smooth, thick, variable=\y, black] plot ({\y*\y-1}, {\y+\y*\y/2-1/2});
  \draw[scale=2.0, domain=-1.0:1.0, smooth, thick, variable=\y, black] plot ({1-\y*\y}, {1/2-\y-\y*\y/2});
\end{tikzpicture}
\caption{With the basis $\B$, $y$ and $z$ have the same first coordinate 
and $S_X$ arrives at $x$ horizontally.}\label{nosuave}
\end{center}
\end{figure}

Let us show that $G$ is not differentiable at 0. To this end, let $z_1'$ and 
$z'_2$ be the coordinates of $\gamma'_z(0)$ in the basis $\B$. With these 
assumptions, $z_1'$ is negative (as in the figure). 
As $G(0)=\lambda,$ we have 
$$G'_-(0)=\lim_{\e\to 0^-}\frac{\|\gamma_z(\e)-y\|_X-\lambda}{\e}\,\cdot$$
As $S_X$ is differentiable at $z$, 
$$\|\gamma_z(\e)-(z+\e\gamma'_z(0))\|_X=\O(\e),$$
so we have 
\begin{equation}
\begin{aligned}
G'_-(0)=&\lim_{\e\to 0^-}\frac{\|z+\e\gamma_z'(0)-y\|_X-\lambda}{\e}= 
\lim_{\e\to 0^-}\frac{\|\lambda x+\e\gamma_z'(0)\|_X-\lambda}{\e}= \\
&\lim_{\e\to 0^-}\frac{\|(0,\lambda)+\e(z'_1,z_2')\|_X-\lambda}{\e}=z'_2,
\end{aligned}
\end{equation}
where the last equality holds because $(0,\lambda)+\e(z'_1,z_2')$ lies in the 
first quadrant and is close to $(0,\lambda)$. In this situation, 
$$\|(0,\lambda)+\e(z'_1,z_2')\|_X=\|(0,\lambda+\e z_2')\|_X+\O(\e)=
\lambda+\e z'_2+\O(\e).$$
So, $G'_-(0)=z'_2$. 

The choice of the basis has nothing to do with the value of $G'_-(0)$. Had 
we computed $G'_+(0)$ by using the basis $\cl{\B}=\{-\gamma'_{x,+}(0),x\}$, 
we would have arrived at $G'_+(0)=\cl{z}'_2$, where $(\cl{z}'_1,\cl{z}'_2)$ 
are the coordinates of $\gamma_z'(0)$ with respect to $\cl{\B}$. What we need 
to see is that $z'_2\neq \cl{z}'_2.$ So, consider the linear automorphism of 
$\R^2$ given by $T(a,b)=(\cl{a},\cl{b})$ when 
$$a\gamma'_{x,-}(0)+bx=\cl{a}\gamma'_{x,+}(0)+\cl{b}x.$$
If we had $\cl{z}'_2=z_2$, then $T$ and the map 
$(a,b)\mapsto(\cl{z}_1'a/z_1,b)$ would agree at $(0,1)$ and $(z_1,z_2)$. 
Both maps are linear and these vectors form a basis, so they must be the same 
map. This readily implies that $S_X$ is differentiable at $x$, a contradiction 
that ends the proof. 
\end{proof}

\begin{remark}
Consider $\R^2$ endowed with the hexagonal norm $\nX$ defined as 
$$\|(a,b)\|_X=\left\{
\begin{array}{l l}
\max\{|a|,|b|\} & \text{ if } ab\geq 0, \\
|a|+|b|         & \text{ if } ab<0
\end{array}
\right.$$
The norm $\nX$ is not differentiable at $x=(0,1)$, but if we take $y=(1,1/3)$ 
and $z=(1,2/3)$ then $\|\gamma_z(t)-y\|_X=1/3+t$ for $t\in[-1/3,1/3]$, so 
$\|\gamma_z(t)-y\|_X$ is differentiable at $t=0$. 
So, if $\nX$ is not strictly convex then Proposition~\ref{Lejos} does not need 
to hold. 
\end{remark}

\begin{questions}
Can Proposition~\ref{Lejos} be generalised to finite-dimensional spaces or arbitrary 
dimension? \\
Is Proposition~\ref{Lejos} true if we replace differentiability by $C^2$-differentiability?
\end{questions}

\begin{nota}
As we will often need to refer to differentiability and non-differentiability 
points, for the sake of readability we will denote $\DD(C_X)$ (resp., $\ND(C_X)$) 
the sets where a curve $C_X$ is differentiable (resp., where it is not differentiable) 
so we will write $x\in \DD(C_X)$ (resp., $x\in \ND(C_X)$) instead of 
{\em $x$ is a differentiability} (resp., {\em non-differentiability}) {\em point of $C_X$.} 
\end{nota}

Before we proceed with our main results, we need these technical Lemmas: 

\begin{lemma}\label{distintos}
Let $C\subset\R^2$ be a Jordan curve that encloses a convex region and 
suppose that the extreme points are all different --that is, there are some 
$c^1=(c_1^1,c_2^1),$ $c^2=(c_1^2,c_2^2),$ $c^3=(c_1^3,c_2^3),$ $c^4=(c_1^4,c_2^4)\in C$ 
such that 
$$c_1^1=\min\{c_1:(c_1,c_2)\in C\}, 
c_2^2=\min\{c_2:(c_1,c_2)\in C\};$$ 
$$c_1^3=\max\{c_1:(c_1,c_2)\in C\}, 
c_2^4=\max\{c_2:(c_1,c_2)\in C\};$$ 
$$
c_1^1<\min\{c_1^2,c_1^3,c_1^4\},\ c_2^2<\min\{c_2^1,c_2^3,c_2^4\},\ 
c_1^3>\max\{c_1^1,c_1^2,c_1^4\},\ c_2^4>\max\{c_2^1,c_2^2,c_2^3\}. 
$$
Then, for any $x=(x_1,x_2)\in\R^2$ there are 
$u=(u_1,u_2),v=(v_1,v_2),w=(w_1,w_2)\in C,$ $t\neq 0$ such that 
$$u-v=tx,\quad w_1=u_1\quad\text{and}\quad w_2=v_2.$$
\end{lemma}

\begin{proof}
If $x=0$ then we may take any $u\in C$ and $w=v=u, t=1.$ So, we only need to 
show that the result holds for $x\neq 0$. It is clear that we may suppose 
$\|x\|_\infty=1$, so we will show that for every $x\in S_\infty$ there exist 
$t,u,v,w$ as in the statement. If $x=(1,0)$ we just need to find $u,v\in C$ that 
belong to the same horizontal line (in this case, $w=u$) 
and if $x=(0,1)$ it suffices to find $u,v\in C$ in the same vertical line and 
take $w=v$, so we may suppose $x_1x_2\neq 0$. 
Suppose that $x_1,x_2>0$, the other cases are analogous. 

As $C$ is convex, there is exactly one point or segment at the undermost end 
of $C$. Suppose it is just one point, say $w^0=c^2$, and analyse what happens 
when we move along the curve anticlockwise until we reach the rightmost point 
or segment of $C$, suppose again that it is a singleton, say $w^1=c^3$. 
As $w^1\neq w^0$, we may consider some anticlockwise 
parameterization of $C$ that has $\gamma(0)=w^0$ and $\gamma(1)=w^1$, we will 
denote $w^s=\gamma(s)$. 

It is clear that, for any $s\in\,]0,1[$, $w^s$ is the undermost point of the 
intersection of $C$ with the vertical line where it lies. Denote $u^s$ the 
uppermost point of this intersection. Analogously, $w^s$ is the rightmost 
intersection of $C$ with the horizontal line where it lies, we will denote its 
leftmost point as $v^s$. 
What we need to see is that for every proportion $x_1/x_2$ there is some $w^s$ 
such that $(w^s_1-v^s_1)/(u^s_2-w^s_2)=x_1/x_2$. 

But the map $s\in\,]0,1[\mapsto (w^s_1-v^s_1)/(u^s_2-w^s_2)$ is continuous and 
its limits are 0 at 0 and $\infty$ at 1. So, at some $s\in\,]0,1[$ 
we get the desired equality. 

If instead of one point there is a segment at the bottom of $C$ then we take 
$w^0$ as the leftmost point of this segment, if there is one segment at the 
rightmost end of $C$ then $w^1$ is at the top of the segment and everything 
goes undisturbed. 
\end{proof}

\begin{lemma}\label{pincho}
Consider $\R^2$ endowed with the norm $\|(\lambda,\mu)\|_1=|\lambda|+|\mu|$.
Let $C\subset\R^2$ be a convex Jordan curve that does not 
fulfil the conditions of Lemma~\ref{distintos} because some point, say $c$, is 
extreme in two directions. For each $a\in C$, consider the sequence 
$(a_n)_n\subset C$ defined as $a_1=a$ and, for $n\geq 1$, $a_{n+1}$ is the 
closest point from $c$ that shares some coordinate with $a_n$. In these  
conditions, $(a_n)_n\to c$ unless $a$ is the strict extreme in the two other 
directions, in which case $a_n=a,\ \forall\, n\in\N$. 
\end{lemma}

\begin{nota}
If there are two possible choices for a given $a_{n+1}$ then we choose 
the point lying in the same vertical line as $a_n$. 
\end{nota}

\begin{proof}
If the conditions in the statement are fulfilled, then it is clear that $(a_n)_n$ 
has some accumulation point because $C$ is compact and the accumulation point 
must be its limit because $(a_n)_n$ is monotonic in both coordinates. 
The only possible limit is $c$, so we are done. 
\end{proof}

\begin{theorem}\label{Piece}
Let $\XnX$ be a two-dimensional normed space whose unit sphere $S_X$ is 
piecewise $C^1$-differentiable and has at least two points $x\neq\pm \cl{x},$ 
with $x, \cl{x}\in \ND(S_X)$. For any normed plane $\YnY$, every isometry 
$\tauSXY$ is linear. 
\end{theorem}

\begin{proof}
If $X$ is not strictly convex then the result holds by~\cite[Corollary~3.8]{JCSRefJMAA}, 
so we may suppose that $\nX$ is strictly convex. 

Suppose there are linearly independent $x,\cl{x}\in \ND(S_X)$, consider the 
basis $\B_X=\{x,\cl{x}\}$, and let $\tauSXY$ be an onto isometry. The only point 
in $S_X$ at distance 2 from $x$ is $-x$, so it is obvious that $\tau(-x)=-\tau(x)$
and we obtain that $\tau(x)$ and 
$\tau(\cl{x})$ are linearly independent so we may consider the basis 
$\B_Y=\{\tau(x),\tau(\cl{x})\}$. 
Taking coordinates with respect to $\B_X$ and $\B_Y$, we have $x=(1,0)_X, 
\tau(x)=(1,0)_Y,$ so both $x$ and $-x$ lie on the same horizontal line and 
$\tau(x)$ and $\tau(-x)$ do, too. 
We are going to see that this happens to $\tau(u),\tau(v)\in S_Y$ for any couple 
$u=(u_1,u_2),v=(v_1,v_2)\in S_X$ such that $u_1>v_1$ and $v_2=u_2$, i.e., such 
that $u-v=\lambda x$ for some $\lambda>0$ --observe that, actually, 
$\lambda=\|u-v\|_X\in ]0,2]$. Let $^\perp x$ be the only point in $S_X$ such 
that $^\perp x\perp_B x$ and whose second coordinate is negative ($^\perp x$ 
is unique because $\nX$ is strictly convex, see \cite[Theorem 4.15]{alonso2012birkhoff}). 
We will denote as $C$ the (relative) interior of the arc of $S_X$ that joins 
$^\perp x$ with $-^\perp x$ and contains $x$, observe that 
$S_X=C\cup-C\cup\{\pm ^\perp x\}$. Consider 
$$\C_{x}=\{u\in C:\tau(u)-\tau(v)=\lambda\tau(x)\text{ if }u-v=\lambda x, 
\lambda\in ]0,2]\}$$
 
We are going to show that $\C_{x}=C$, so we will have the equivalence 
$$u-v=\lambda x \Leftrightarrow \tau(u)-\tau(v)=\lambda\tau(x).$$ 

If $(u_n)_n\to u\in C$ and $u_n\in\C_{x}$ for every $n\in\N$, then consider 
the corresponding sequences $(v_n)_n,(\lambda_n)_n$. We have, for every $n\in\N$, 
\begin{equation}\label{sucesiones}
u_n-v_n=\lambda_n x\text{\quad and\quad }\tau(u_n)-\tau(v_n)=\lambda_n\tau(x). 
\end{equation}
It is clear that both $(v_n)_n,(\lambda_n)_n$ must converge and that
$$u-v=\lambda x \quad \text{and} \quad \tau(u)-\tau(v)=\lambda\tau(x),$$
where $\lambda=\lim(\lambda_n),v=\lim(v_n)$. So, $u\in\C_{x}$ and
$\C_{x}$ is, therefore, closed in $C$.

Suppose now that $(u_n)_n\to u$ and $u\in\C_{x}$. Take $\lambda, v$ and 
$(\lambda_n)_n, (v_n)_n$ such that $u-v=\lambda x, u_n-v_n=\lambda_n x$. 
As there are finitely many points in $\ND(S_X)$, we may suppose that 
none of them is $u_n$ or $v_n$. In this situation, Proposition~\ref{Tlocal} 
implies that $S_Y$ is differentiable at $\tau(u_n)$ and so, Proposition~\ref{Lejos} 
implies that $\tau(u_n)-\tau(v_n)=\lambda_n y_n$, with $y_n\in \ND(S_Y)$. 
As $\ND(S_Y)$ is finite, there is some $y$ that appears infinitely many 
times in $(y_n)_n$, so passing to a subsequence we may suppose that 
$\tau(u_n)-\tau(v_n)=\lambda_n y$ for every $n\in\N$. Of course, 
$\lim(\tau(u_n))_n=\tau(u)$, $\lim(\tau(v_n))_n=\tau(v)$ and 
$\lim(\lambda_n)_n=\lambda$, so 
$$\lambda y=\lim(\tau(u_n)-\tau(v_n))_n=\tau(u)-\tau(v)=\lambda\tau(x),$$ 
we obtain that $y=\tau(x)$ and this means that $\C_{x}$ is open. 

We have seen that $\C_{x}$ is non-empty --because $x\in\C_{x}$--, 
closed and open, so the connectedness of $C$ shows that $\C_{x}=C$. 

So, $u-v=\lambda x$ implies $\tau(u)-\tau(v)=\lambda\tau(x)$. Of course, the same 
applies to $\cl{x}$, so what we actually have is that 
$u-v=\lambda x+\mu \cl{x}$ implies $\tau(u)-\tau(v)=\lambda\tau(x)+\mu\tau(\cl{x})$
whenever there exists $w\in S_X$ such that either $w=v+\mu \cl{x}=u-\lambda x$ 
or $w=v+\lambda x=u-\mu \cl{x}$. 

Now we have two options. If we are in the hypotheses of Lemma~\ref{distintos}, 
then $w$ exists for every possible direction, 
and from the fact that $\tau$ is an isometry, we get 
$$\|\lambda\tau(x)+\mu\tau(\cl{x})\|_Y=\|\lambda x+\mu \cl{x}\|_X.$$ 
As we have taken coordinates with respect to $\{x,\cl{x}\}$ and 
$\{\tau(x),\tau(\cl{x})\}$, we get $\|(\lambda,\mu)\|_Y=\|(\lambda,\mu)\|_X$. 
This means that in these coordinates we have $\nY=\nX$, 
so~\cite[Theorem~2.3]{JCSRefJMAA} implies that $\tau$ is linear. 

If we cannot apply Lemma~\ref{distintos}, then  there is some $c\in S_X$ that is 
strictly extremal in two different directions. So, we may apply Lemma~\ref{pincho} 
to show that, given any $a\neq\pm c\in S_X$, the 
sequence $(\tau(a_n))_n$ is the same as the sequence $((\tau(a))_n)_n$, i.e., the 
sequence originated in $\tau(a)$. This implies that for every $n\in\N$, and 
with the coordinates taken again with respect to $\{x,\cl{x}\}$ and 
$\{\tau(x),\tau(\cl{x})\}$ we have $\tau(a_n)-\tau(a)=a_n-a$. As $(a_n)_n \to c$, 
this means that for every $a\in C,$ we have $\tau(a)-\tau(c)=a-c$. From 
here we readily see that $\tau$ is linear in $S_X$ and this completes the proof. 
\end{proof}

\begin{remark}
If the only non-differentiability points in $S_X$ are $\pm x$, then it is clear 
from the previous proof that $u-v=\lambda x$ implies $\tau(u)-\tau(v)=\lambda\tau(x)$, 
but we have not been able to infer from here that $\tau$ must be linear. Taking any 
basis $\B_X=\{x,\cl{x}\}$ and considering $\B_y=\{\tau(x),\tau(\cl{x})\}$ we 
have, in coordinates, $\tau(\alpha,\beta)-\tau(\alpha',\beta)=(\alpha-\alpha',0)$ 
for every $(\alpha,\beta), (\alpha',\beta)\in S_X$ but we have not been able to 
deduce anything for points with different second coordinates. 
\end{remark}

\begin{theorem}[Mankiewicz Property]\label{coroultra}
Let $\XnX, \YnY$ be strictly convex normed planes such that both $\ND(S_X)$ and 
$\ND(S_Y)$ are finite and contain more than three points. 
Let $C_X\subset X$ be a piecewise $C^1$ Jordan curve that encloses a convex set.
If there is some isometry $C_X\to C_Y$ for some piecewise $C^1$ curve
$C_Y\subset Y$, then $\tau$ is affine and $X$ and $Y$ are isometric.
\end{theorem}

\begin{proof}
First of all, we hasten to remark that in strictly convex spaces, if three 
points $z^1, z^2, z^3$ fulfil $\|z^1-z^3\|=\|z^1-z^2\|+\|z^2-z^3\|$, then 
$z^2$ belongs to the segment whose endpoints are $z^1$ and $z^3$, we will 
denote this segment as $[z^1,z^3]$. With this, it is not hard to see that 
a curve $C$ encloses a convex region if and only if for every triple of 
collinear points $z^1, z^2, z^3\in C$, the curve $C$ contains the segment 
$[z^1,z^3]$. So, our hypotheses imply that $C_Y$ is convex, too. 

The first we need to show is that Corollary~\ref{Clocal} and Proposition~\ref{Lejos} 
still apply in this situation, i.e., that if $u^0,v^0\in C_X$ fulfil that 
$(u^0-v^0)/\|u^0-v^0\|_X\in \ND(S_X)$, then for every $u,v\in C_X$ we have 
the equivalence $u-v=\lambda (u^0-v^0)$ if and only if 
$\tau(u)-\tau(v)=\lambda(\tau(u^0)-\tau(v^0))$. 

For the equivalent of Corollary~\ref{Clocal}, we have to make do without $-x$, 
but the only thing really useful of having $-x$ was that $x\in \DD(S_X)$ if and 
only if $-x\in \DD(S_X)$. In any case, we are going to show that $x\in \DD(C_X)$ 
is equivalent to $\tau(x)\in \DD(C_Y)$. As the proof is going to be quite 
different, we will denote the point as $a$ instead of $x$. 

Let $a\in C_X$ and let us analyse the set 
$$\NDif(a)=\{b\in C_X:\|\gamma_a(t)-b\|_X\text{ is not differentiable at }t=0\}, $$
observe that one has $a\in\NDif(a)$ for every $a\in C_X$. 

If $a\in\DD(C_X)$, then it is clear that $a\neq b\in\NDif(a)$ implies 
$(a-b)/\|a-b\|_X\in \ND(S_X)$ no matter whether $b\in\DD(C_X)$ or not. 
As there are only finitely many points in 
$\ND(S_X)$, say $\ND(S_X)=\{x^1,\ldots,x^m\}$, when $b\in\NDif(a)$ one has 
$b-a=\|b-a\|_Xx^i$ for some $i\in\{1,\ldots,m\}$ --we are considering as 
unrelated points $x^i$ and $-x^i$. 

\begin{claim}
For any $a\in\DD(C_X)$, $\NDif(a)$ contains, at most, one segment and finitely 
many isolated points. If it contains one segment, then one of its endpoints is $a$. 
\end{claim}

\begin{proof}
We need to show that for every $i\in\{1,\ldots,m\}$, there is at most one point 
in $\NDif(a)$ that can be written as $b-a=\|b-a\|_Xx^i$ unless there is a 
segment that fulfils it. Indeed, if $b^1, b^2$ fulfil $b^1-a=\|b^1-a\|_Xx^i$ 
and $b^2-a=\|b^2-a\|_Xx^i$ with $\|b^1-a\|_X<\|b^2-a\|_X$, then $b^1$ lies in 
the interior of the segment $[a,b^2]$ --i.e., the closed segment whose endpoints 
are $a$ and $b^2$. As $a, b^1, b^2\in C_X$ and $C_X$ encloses a convex region, 
the segment $[a,b^2]$ is included in $C_X$ and it is obvious that there is only 
one segment included in $C_X$ that has $a$ as its endpoint --recall that $C_X$ 
is differentiable at $a$. If $a$ is interior to some segment, then no more 
segments can arrive to $a$ and it is clear that $\|\gamma_a(t)-b\|_X$ is 
differentiable at 0 for any point in the same segment. 
\end{proof}

If we have, instead, $a\in \ND(C_X),$ then $\NDif(a)$
includes every $b\in C_X$ such that $(a-b)/\|a-b\|_X\in \DD(S_X)$. Indeed, 
let $x=(a-b)/\|a-b\|_X\in \DD(S_X)$ and consider $\bar{x}$ as any of the two 
opposite vectors in $S_X$ such that $x\perp_B \bar{x}$ --i.e, 
$\bar{x}=\pm \gamma'_x(0)\in S_X$. With the basis $\B_X=\{\bar{x},x\}$, the 
sphere $S_X$ and the line $\{(\lambda,1):\lambda\in\R\}$ 
are tangent. This implies that, for every $\mu\in\,]\!-1,1[$, the line 
$\{(\lambda,\mu):\lambda\in\R\}$ meets $S_X$ in two points, say 
$b=(b_1,b_2), c=(c_1,c_2)$, and the first coordinates of these points 
have different sign. Moreover, as $\nX$ is strictly 
convex, $\{(\lambda,1):\lambda\in\R\}\cap S_X=\{x\}$. Both these facts will 
be important later. 

Denote $a'_-, a'_+$ the 
(different) side derivatives of $C_X$ in $a$. In the basis $\B_X=\{\bar{x},x\}$, we 
have $a'_-=(a'_{-,1},a'_{-,2}), a'_+=(a'_{+,1},a'_{+,2}).$ The speed of growing 
of $\|\gamma_a(t)-b\|_X$ as we are arriving at $a$ in the direction of $a'_-$ 
is $a'_{-,2}$ and the speed of growing in the direction of $a'_+$ is 
$a'_{+,2}$. This can be seen as in Proposition~\ref{Lejos} or by thinking 
this situation as if we had partial derivatives: $\|\gamma_a(t)-b\|_X$ would 
grow at speed 1 if $\gamma'_a(0)=x=(0,1)$ and the speed would be 0 if 
$\gamma'_a(0)=\bar{x}=(1,0)$. For any linear combination $(a_1,a_2)$ we have speed $a_2$. 
So, we need to show that $a'_{-,2}\neq a'_{+,2}$. As $C_X$ encloses a convex 
region, the signs of $a'_{-,1}$ and $a'_{+,1}$ are the same --maybe one of them 
is zero. So, if $a'_{-,2}= a'_{+,2}$, then we would have two points in $S_X$ with 
the same second coordinate in the same quadrant, but we have just seen that 
this cannot happen. 

This means that $\NDif(a)$ contains every point in $C_X$ but, at most, 
two segments that include $a$ and finitely many other points. 
In particular, there exist some open $U\subset C_X$ 
such that $U\subset\NDif(a)$ and $U\cup\{a\}$ is not contained in a metric segment. 

Gathering all these facts, we obtain that the metric structure of $\NDif(a)$ 
determines the differentiability of $C_X$ at $a$ --and this implies that 
$\tau(a)\in \DD(C_Y)$ if and only if $a\in \DD(C_X)$. 

As for the analogous of Proposition~\ref{Lejos}, we need to show that 
$(u-v)/\|u-v\|_X\in\ND(S_X)$ implies 
$(\tau(u)-\tau(v))/\|\tau(u)-\tau(v)\|_Y\in\ND(S_Y)$.

So, let $x=(u-v)/\|u-v\|_X\in\ND(S_X)$ and consider, as in the proof of 
Proposition~\ref{Lejos}, the basis $\B=\{-\gamma'_{x,-}(0),x\}$. We have
$u$ and $v$ in the same vertical line and $u$ is over $v$. 
In coordinates, $u_1=v_1, u_2>v_2$. 

Suppose that $u\in\DD(C_X)$ and that there is no segment in $C_X$ that contains 
$u$ and $v$. Then, the map 
$$G(t)=\|\gamma_u(t)-v\|_X$$ 
is not differentiable at $t=0$, the proof is the same as the one in 
Proposition~\ref{Lejos}. 

In the proof of Theorem~\ref{Piece}, we defined $C$ as the interior of one of 
the arcs that join $^\perp x$ and $-^\perp x$. The analogous way to define this 
is by taking $C$ as the interior of the arc that joins the lowermost point or 
segment in $C_X$ with its uppermost point or segment passing through the right 
part of $C_X$ --so, $C$ does not include any of its endpoints. 

Later, we defined
$$\C_{x}=\{u\in C:\tau(u)-\tau(v)=\lambda\tau(x)\text{ if }u-v=\lambda x, 
\lambda>0\}, $$
but now we do not have $\tau(x)$, so we need to modify the definition of the 
set $\C_x$. For this, there is an equivalent way to state $u-v=\lambda x$ and 
$\tau(u)-\tau(v)=\lambda\tau(x)$. 
We can take $u^0\in C,\ v^0\in C_X$ such that $u^0-v^0=\lambda_0 x$ with $\lambda_0>0$, 
eliminate the condition $u-v=\lambda x$ by writing $u-\lambda x$ instead of $v$ 
and define our new subset as any of the following equivalent ways: 
$$\C_{x}=\{u\in C:\tau(u)\!-\!\tau(u\!-\!\lambda x)=
\lambda(\tau(u^0)\!-\!\tau(v^0))/\lambda_0,\ \lambda>0\}, $$
$$\C_{x}=\{u\in C:\tau(u)\!-\!\tau(u\!-\!\lambda x)=
\lambda(\tau(u^0)\!-\!\tau(u^0\!-\!\lambda_0 x))/\lambda_0,\ \lambda>0\}. $$
Now, the proof of Theorem~\ref{Piece} shows that $\C_{x}$ is open and closed in $C$. 
Again, $\C_{x}$ is not empty because $u^0\in\C_x$, so $\C_x=C$. It is clear that 
if for every $u$ in $C_X$ such that there is exactly one $v\in C_X$ such that 
$u-v=\|u-v\|_Xx$ one has $\tau(u)-\tau(v)=\lambda(\tau(u^0)-\tau(v^0))$ for some 
$\lambda>0$, we have the same when $u$ belongs to a segment whose direction 
is $x$. So, denoting $y=(\tau(u^0)-\tau(v^0))/\|\tau(u^0)-\tau(v^0)\|_Y$ 
we have $u-v=\|u-v\|_Xx$ if and only if $\tau(u)-\tau(v)=\|\tau(u)-\tau(v)\|_Yy$. 

If we consider $\cl{x}\in \ND(S_X),\cl{x}\neq \pm x$ and 
$\cl{u}^0, \cl{v}^0\in C_X$ such that $\cl{u}^0-\cl{v}^0=\|\cl{u}^0-\cl{v}^0\|_X\cl{x}$, 
then the same argument as before shows that $\cl{u}-\cl{v}=\|\cl{u}-\cl{v}\|_X\cl{x}$
is equivalent to $\tau(\cl{u})-\tau(\cl{v})=\|\tau(\cl{u})-\tau(\cl{v})\|_Y\cl{y}$, 
with $\cl{y}=(\tau(\cl{u}^0)-\tau(\cl{v}^0))/\|\tau(\cl{u}^0)-\tau(\cl{v}^0)\|_Y$.

With this, if we consider the bases $\{x,\cl{x}\}$ and $\{y,\cl{y}\}$, we have 
the equivalences 
$$u-v=(\|u-v\|_X,0) \text{ if and only if } \tau(u)-\tau(v)=(\|u-v\|_X,0),$$
$$u-v=(0,\|u-v\|_X) \text{ if and only if } \tau(u)-\tau(v)=(0,\|u-v\|_X).$$

Now we have two options: if we can apply Lemma~\ref{pincho} then the remainder
of the proof goes as the last part of the proof of Theorem~\ref{Piece}.
Otherwise, we can apply Lemma~\ref{distintos} to obtain that, in the bases 
$\{x,\cl{x}\}$ and $\{y,\cl{y}\}$, we have $\nX=\nY$. It remains to show that 
$\tau$ is affine. 

Now, we may suppose that 
$Y=X$ and we need to show that $\tau(a)-\tau(b)=a-b$ for every $a,b\in C_X$. 

In what followswe suppose that $C_X$ has no horizontal nor vertical segment. 
An analogous idea gives a proof for the other cases. 
Observe that if $a$ and $b$ belong to the same horizontal or vertical segment, 
then we have $\tau(a)-\tau(b)=a-b$. We will denote by $W, S, E$ and $N$ 
respectively the leftmost, the undermost, rightmost and uppermost points in $C_X$. 
We will also denote as $SW, SE, NE$ and $NW$ the closed arcs that join each pair 
of consecutive extremal points. 

We will denote $\E_a=\{b\in C_X:\tau(a)-\tau(b)=a-b\}$. So, we need to show 
that $\E_a=C_X$ for some (every) $a\in C_X$. 

For any $a\in C_X$, denote $a^0=a;$ $a^1\in C_X$ is the other point that lies 
in the same horizontal line as $a^0$, $a^2\in C_X$ lies in the same vertical line 
as $a^1$ and so on. Furthermore, let $a^{-1}\in C_X$ be the point that lies 
in the same vertical line as $a^0$, $a^{-2}\in C_X$ is in the same 
horizontal line as $a^{-1}$\ldots This {\em bi-infinite sequence} may hit 
some extremal point and, so to say, get stuck --but this changes nothing. 
Moreover, if $C_X$ has some vertical or horizontal symmetry, then 
$a^4=a^0=a^{-4}$. Still, everything goes fine. It is clear that, given 
$n,m\in\Z$, one has $\tau(a^n)-\tau(a^m)=a^n-a^m$. The continuity of the 
isometry is enough to ensure 
$$\tau(\lim(a^{n_k}))-\tau(\lim(a^{m_k}))=\lim(a^{n_k})-\lim(a^{m_k})$$
for any convergent subsequences $(a^{n_k}), (a^{m_k})$. So, 
$\tau(c^1)-\tau(c^2)=c^1-c^2$ whenever $c^1,c^2\in\cl{\{a^n:n\in\Z\}}$, 
i.e., $\cl{\{a^n:n\in\Z\}}\subset \E_a$. 

To end the proof we need three more facts. 

\begin{claim}\label{4Q}
For $a,b\in C_X$, if there are $i_1,i_2,i_3,i_4,j_1,j_2,j_3,j_4\in\Z$ such that 
each $a^{i_k}-b^{j_k}$ lies in the $k$-th quadrant, then $b\in\E_a$ or, equivalently, 
$\E_b=\E_a$.
\end{claim}

\begin{proof}
We may suppose, after composition with some translation, that $\tau(a)=a$. If $\tau(a)-\tau(b)\neq a-b$, 
then there is some $v\neq 0$ such that $\tau(b^n)=b^n+v$ for every $n\in\Z$. 
As $\nX$ is strictly convex, there is a half-plane $H$ such that $\|u-v\|_X>\|u\|_X$ 
whenever $u\in H$. Namely, with $w=^\perp\!\! v$, 
$H=\{\alpha v+\beta w:\alpha\in\,]\!-\infty,0],\beta\in\R\}$. In particular, there is a whole 
quadrant included in $H$, so for some $k\in\{1,2,3,4\}$ one has 
$$\|\tau(a^{i_k})-\tau(b^{j_k})\|_X=\|a^{i_k}-b^{j_k}-v\|_X>\|a^{i_k}-b^{j_k}\|_X,$$ 
a contradiction with the fact that $\tau$ is an isometry. 
\end{proof}

\begin{claim}\label{SW}
For every $a\in C_X$, the intersection of ${\{a^n:n\in\Z\}}$ with every arc 
$SW, SE, NE$ and $NW$ is nonempty. 
\end{claim}

\begin{proof}
Let us see that $SW\cap {\{a^n:n\in\Z\}}\neq\emptyset$, the other cases follow by 
symmetry. As $C_X$ is convex, some point $b=(b_1,b_2)$ belongs to the arc $SW$ 
if and only if there is no point $c=(c_1,c_2)\in C_X$ such that $c_1\leq b_1, 
c_2\leq b_2$ and $c\neq b$. Given $a\in C_X\setminus SW$, either $a^1$ or 
$a^{-1}$ --maybe both-- has a coordinate that is smaller than that of $a$, say 
$a^1_1<a_1$. If $a^1\not\in SW$, then $a^2_2<a^1_2$ and so on. 
If there is no $n$ such that $a^n\in SW$, then the sequence has an accumulation 
point, but this accumulation point must be the limit of the sequence because 
the sequence is bounded and nonincreasing in both coordinates. So, we are in 
the conditions of Lemma~\ref{pincho}, a contradiction.
\end{proof}

\begin{claim}\label{S}
For $S=(S_1,S_2)$ --the undermost point of $C_X$-- we have either $SW\subset\E_S$ 
or $SE\subset\E_S$. 
\end{claim}

\begin{proof}
Let $b=(S_1,b_2)\in C_X$ be the other point with the same first coordinate as $S$. 
It is clear that $b\in\E_S$. Given $c^1=(c^1_1,c^1_2)\in C_X$ with $c^1_2<b_2$, take 
$c^2=(c^1_1,c^2_2)\in C_X$ with, say, $c^1_1<c^2_1$. We have $c^2\in\E_{c^1}$ and 
moreover $b-c^1, b-c^2, S-c^2, S-c^1$ are, respectively in the first, second, third 
and fourth quadrants. Claim~\ref{4Q} implies that $c^1, c^2\in\E_S$ so, for every 
$c=(c_1,c_2)\in C_X$ such that $c_2\leq b_2$ one has $c\in\E_S$. As $C_X$ encloses 
a convex region, $E_1<b_1<W_1$ implies $b_2\geq\min\{E_2,W_2\}$. Now we may suppose 
$E_2\leq b_2$, that implies $u_2\leq b_2$ for every $(u_1,u_2)\in SE$, so $SE\subset\E_S$. 
\end{proof}
Now we just need to use Claims~\ref{SW} and~\ref{S} to see that $\E_S=C_X$, so 
we have finished the proof. 
\end{proof}

\begin{remark}
After Theorem~\ref{Piece},~\cite[Theorem~1.5]{TBanakhC2} 
and~\cite[Corollary~3.8]{JCSRefJMAA}, the only possibility for the existence 
of a nonlinear isometry between two-dimensional spheres is that both of them 
are strictly convex and one of the following holds: 
\begin{itemize}
\item Both spheres are $C^1$-differentiable and at least one of them is not 
absolutely smooth. 
\item None of the spheres is piecewise differentiable, i.e., there are 
infinitely many points of non-differentiability in each sphere. 
\item Both spheres have exactly two points of non-differentiability, say, $x$ and $-x$. 
\end{itemize}
\end{remark}

\section*{Acknowledgements}
I would like to thank Professor Tarás Banakh and my colleagues Daniel Morales 
and José Navarro for some valuable discussions regarding the two-dimensional 
Tingley's Problem. 

It is a pleasure to thank Professor Javier Alonso for the Gift~\ref{Alonso}. 

I absolutely need to thank the anonymous referee for their fantastic reports. 
The work reads much better because of these reports and, in particular, 
the proofs of Proposition~\ref{Tlocal}, Proposition~\ref{Lejos} and 
Theorem~\ref{coroultra} owe this referee a great debt. 

Supported in part by Junta de Extremadura programs GR-15152 and IB-16056 and 
DGICYT projects MTM2016-76958-C2-1-P and PID2019-103961GB-C21 (Spain).

\section*{References}
\bibliographystyle{abbrv}

\addcontentsline{toc}{section}{Bibliography}
\bibliography{DifferentiabilityMUP}{}

\end{document}